\documentclass[11pt, reqno]{amsart}  
\usepackage{latexsym}
\usepackage[utf8]{inputenc}
\usepackage[english]{babel}
\usepackage{amsmath}
\usepackage{amsfonts}
\usepackage{amssymb}
\usepackage{graphicx}
\usepackage{enumerate}
\usepackage{url}
\usepackage{bm} 
\usepackage{verbatim}
\usepackage{cite}
\usepackage{multicol}
\usepackage{color}

\usepackage[colorlinks=true]{hyperref}
\hypersetup{urlcolor=blue, citecolor=blue}

\renewcommand{\geq}{\geqslant}
\renewcommand{\leq}{\leqslant}

\newtheorem{theorem}{Theorem}

\newtheorem{corollary}{Corollary}
\newtheorem{lemma}{Lemma}

\newtheorem{remark}{Remark}
\newtheorem{example}{Example}

\newcommand{\R}{\mathbb{R}}

\newcommand{\N}{\mathbb{N}}
\newcommand{\C}{\mathbb{C}}

\def\adj{\operatorname{adj}}

\title[A new family of integrable differential systems]{A new family of integrable differential systems in arbitrary dimension}

\author[J. D. Garc\'\i a--Salda\~na]{Johanna D. Garc\'\i a--Salda\~na}
\address{Departamento de Matem\'atica y F\'\i sica Aplicadas, Universidad Cat\'olica de la Santis\'\i ma Concepci\'on, Concepci\'on, Chile}
\email{jgarcias@ucsc.cl}

\author[A. Gasull]{Armengol Gasull}
\address{Departament de Matem\`atiques, Universitat Aut\`onoma de Barcelona, Barcelona, Spain}
\email{armengol.gasull@uab.cat}

\author[S. Rebollo--Perdomo]{Salom\'on Rebollo--Perdomo}
\address{Departamento de Matem\'atica, Universidad del B\'\i o-B\'\i o, Concepci\'on, Chile}
\email{srebollo@ubiobio.cl}

\date{28/01/2025}

\begin{document}

\begin{abstract} We present a wide class of differential systems in any dimension that are either integrable or complete integrable. In particular, our result enlarges  a known family of planar integrable systems.  
We give an extensive list of examples that illustrates the applicability of our approach. For instance, in the plane this list includes some  Liénard, Lotka--Volterra and quadratic systems; in the space, some  Kolmogorov, Rikitake and Rössler systems. Examples of complete integrable systems in higher dimensions are also provided.
\end{abstract}

\maketitle

\section{Introduction and main result}

This work deals with $C^r\!$ real autonomous differential systems, {\it i.e.}, systems of ordinary differential equations 
\begin{equation}
\label{sistema-autonomo-general}
\dot{\mathbf{x}}=\mathbf{F}(\mathbf{x}), \qquad  \mathbf{x}\in 
\Omega \subset \R^n,
\end{equation}
where the dot denotes derivative with respect to an independent real variable~$t$,
$\Omega$ is an open subset of $\R^n$ 
and
$\mathbf{F}$ 
is a $C^r\!$ function in $\Omega$, 
with $r \in \N\cup \{\infty,\omega\}$, where as usual $\N$ denotes the set of positive integer numbers. 
Recall that $C^{\omega}\!$
stands for analytic functions.

A $C^k\!$ \emph{first integral} 
in $\Omega$ of system 
\eqref{sistema-autonomo-general}
is a locally non-constant $C^k\!$ function 
$H\colon \Omega \longrightarrow \R$, with $k \in \N \cup \{\infty,\omega\}$, that satisfies the equation
\begin{equation*}
   \dot{H}(\mathbf{x})= \nabla H(\mathbf{x})\, \mathbf{F}(\mathbf{x}) = 0
\end{equation*}
on the whole $\Omega$. 
This means that $H$ is constant throughout the orbits of the 
system. In other words, 
the orbits of the 
system live in the level hypersurfaces of the first integral $H,$ which implies that we can reduce in one dimension the study of \eqref{sistema-autonomo-general}.

Recall that 
the $C^k\!$ functions 
$H_1(\mathbf{x}), \ldots, H_m(\mathbf{x})$
are \emph{functionally independent in $\Omega$}
if their gradients, 
$\nabla H_1(\mathbf{x}),
\ldots,
\nabla H_m(\mathbf{x})
$,
are linearly independent in a 
full Lebesgue measure open subset of $\Omega$.
Hence, if system \eqref{sistema-autonomo-general}
has $m$ functionally independent 
$C^k\!$ first integrals in~$\Omega$, 
with $1\leq m<n$, then 
the orbits of the  system live in
the intersection set of the level hypersurfaces 
of these first integrals, which generically is a $(n - m)$-dimensional subset of~$\Omega$. If $m = n - 1$, then the intersection set is in general one-dimensional, and so it is formed by orbits of the system. 
Thus, if system \eqref{sistema-autonomo-general} has $n - 1$ functionally
independent first integrals, then all orbits of the system are determined by the
first integrals except perhaps in a zero Lebesgue measure subset of~$\Omega$. Therefore, it is said that
system  
\eqref{sistema-autonomo-general} 
is $C^k\!$ \emph{completely integrable in $\Omega$}
if it has $n-1$ functionally independent 
$C^k\!$ first integrals in $\Omega$.

The \emph{weak}  or \emph{strong} versions of the integrability problem consist in finding conditions on $\mathbf{F}(\mathbf{x})$ under which   system \eqref{sistema-autonomo-general} has one first integral or  $n-1$ functionally independent first integrals, respectively. Both versions also ask about the type of such first integrals: are they  $C^k\!$ functions, with $k \in \N \cup \{\infty,\omega\}$, polynomial, rational, Darboux, or Liouville? 
In this sense, differential system  
\eqref{sistema-autonomo-general} 
is called \emph{polynomially integrable in $\Omega$} if it has a polynomial first
integral in $\Omega$ and it is called \emph{polynomially completely integrable in $\Omega$} if it has
$n-1$ functionally independent  polynomial  first integrals in $\Omega$. Similarly, we can define rational, Darboux, or Liouville (completely) integrable in $\Omega$.

The integrability of differential systems \eqref{sistema-autonomo-general} is a very active research issue in the field of ordinary differential equations. See for instance \cite{AGG2019,AGR2019,LRR2020} and references there in, or next discussion. 
There are different tools and methods to study the integrability of differential systems; for example, Lie symmetries \cite{Olv1993}, Noether symmetries \cite{SaCa1981}, the Darboux theory of integrability \cite{Dar1878a, Dar1878b, LlZh2009}, Lax pairs \cite{Lax1968,Gor2001,LlZh2009}, singularity analysis \cite{Gor2001}, and integrating factors as well as inverse integrating factors \cite{GGG2010}.
However, in general, the problem of determining whether a concrete differential system \eqref{sistema-autonomo-general} is integrable remains unsolved.
 
In this work, we will focus in 
$C^r\!$ real autonomous differential systems of the form 
\begin{equation}
\label{sistema-autonomo-local}
\dot{\mathbf{x}}=\mathbf{F}(\mathbf{x}), \qquad  \mathbf{x}\in 
(\R^n,\mathbf{0}),
\end{equation}
where 
$(\R^n,\mathbf{0})$ 
is a small enough 
open neighborhood of 
the origin $\mathbf{0}$ 
of $\R^n.$
We are mainly interested in
the local strong integrability problem, that is, 
to find a particular structure on 
$\mathbf{F}(\mathbf{x})$ 
under which system \eqref{sistema-autonomo-local} is (completely) integrable in  
$(\R^n,\mathbf{0})$.
We shall assume that 
$\mathbf{0}$ is a singular point of the differential system, that is, 
$\mathbf{F}(\mathbf{0})=\mathbf{0}$, because otherwise the classical flow box theorem solves the problem.

Recall that $\adj(\mathbf {A})$ denotes the adjoint of a square matrix $\mathbf{A}$ and it is the transpose of the cofactor matrix of $\mathbf {A}$. If $A$ is  of  size $n$, then the identity 
$
\mathbf {A} \operatorname {adj} (\mathbf {A} )=\det(\mathbf {A} )\mathbf{I}_n
$
holds,
where $\mathbf{I}_n$ is 
the identity matrix of size $n$. Moreover, in particular, when $\mathbf {A}$ is invertible $\mathbf {A}^{-1}=\operatorname {adj} (\mathbf {A} )/\det(\mathbf {A} ).$

Before giving our main result we explain the simple idea that leads to its statement.  It is well-known that if  two differential systems, both defined  in $ 
(\R^n,\mathbf{0}),$
\[
\dot{\mathbf{x}}=\mathbf{K}(\mathbf{x}) \quad\mbox{and}\quad  \dot{\mathbf{u}}=\mathbf{G}(\mathbf{u}),  
\] 
are conjugated  via a $C^k\!$ diffeomorphism $\mathbf{u}=\mathbf{\Phi} (\mathbf{x}),$ $\mathbf{x}=\mathbf{\Phi}^{-1} (\mathbf{u}),$ such that $\mathbf{\Phi}(\mathbf{0})=\mathbf{0}$, then they share all their dynamical properties. In particular, they have the same number of functionally independent first integrals. Let us specify the relation between both differential systems:
\begin{align*}
\mathbf{K}(\mathbf{x}) &=D \mathbf{\Phi}^{-1} (\mathbf{\Phi} (\mathbf{x})) \mathbf{G}(\mathbf{\Phi} (\mathbf{x}))=\big(D \mathbf{\Phi}(\mathbf{x}))\big)^{-1} \mathbf{G}(\mathbf{\Phi} (\mathbf{x})) \\ &=\frac{ \operatorname {adj} (D \mathbf{\Phi}(\mathbf{x})))}{\det (D \mathbf{\Phi}(\mathbf{x}))}    \mathbf{G}(\mathbf{\Phi} (\mathbf{x})).
\end{align*}
Moreover, if in the differential system  $\dot{\mathbf{x}}=\mathbf{K}(\mathbf{x})$ we change the time $t$ by a new one, $s,$ such that ${{\rm d}t }/{{\rm d}s}= R(\mathbf{x})\det (D \mathbf{\Phi}(\mathbf{x})),$ for some  real-valued function $R$ that is defined in an open subset $\mathcal{O}$ of $(\R^n, \mathbf{0})$ and it is nonzero on open sets, then  
we  get the new differential system
\[
\mathbf{x}'=\frac{{\rm d}t }{{\rm d}s}\dot{\mathbf{x}}=R(\mathbf{x})\operatorname {adj} (D \mathbf{\Phi}(\mathbf{x})))\mathbf{G}(\mathbf{\Phi} (\mathbf{x})),
\] 
which is equivalent to  $\dot{\mathbf{u}}=\mathbf{G}(\mathbf{u})$ in $\mathcal{O} \backslash \{R(\mathbf{x})=0\}$. It is also  well-known that this new differential system has the same number of functionally independent first integrals that  $\dot{\mathbf{u}}=\mathbf{G}(\mathbf{u}).$

Although the derivation of the last differential system uses that $\mathbf{\Phi}$  is invertible, its structure remains well-defined even when $\mathbf{\Phi}$ is not invertible.
This observation serves as the basis for conjecturing that completely integrable differential systems can be constructed by using non-invertible functions $\mathbf{\Phi}$. This is precisely what our main result will establish.

We also must introduce a key concept in our approach. Let 
$\bm{\Phi} \colon (\mathbb{R}^n, \mathbf{0}) \longrightarrow \mathbb{R}^n$ 
be a $ C^k\!$ map, with $ k \in \mathbb{N} \cup \{ \infty, \omega \} $. We say that $ \bm{\Phi} $ \emph{preserves dimension} if the image $ \bm{\Phi}(U) $ of any open set $ U \subset (\mathbb{R}^n, \mathbf{0}) $ contains open sets. Of course, diffeomorphisms preserve dimension, but other non-injective or non-exhaustive  maps, like for instance the 2-fold map  $\bm{\Phi}(x,y)=(x^2,y)$, do as well.

\begin{theorem}\label{th:main} Consider the $C^r\!$  differential system,	with $r\in \N \cup \{\infty,\omega\},$
	\begin{equation}
		\label{main-system}
		\dot{\mathbf{x}}=\mathbf{F}(\mathbf{x}):= R (\mathbf{x})
		\adj(D\bm{\Phi}(\mathbf{x}))\,
		\mathbf{G}(\bm{\Phi}(\mathbf{x})), \qquad  \mathbf{x}\in 
		(\R^n,\mathbf{0}),
	\end{equation}
	with $\mathbf{F}(\mathbf{0})=\mathbf{0}$, 
where   $\bm{\Phi}, \mathbf{G} \colon
(\R^n,\mathbf{0}) \longrightarrow \R^n$ are suitable $C^k\!$ functions, with $k\ge r$, $\bm{\Phi}(\mathbf{0})=\mathbf{0}$, and $R(\mathbf{x})$ is a real-valued function   such that  	 $R(\mathbf{x})\neq 0$ in a full Lebesgue measure open subset of $(\R^n,\mathbf{0})$. Assume also that $\bm{\Phi}$ preserves dimension.  Then:
\begin{enumerate}[(i)]
	
\item  If $I\colon (\R^n,\mathbf{0})\longrightarrow \R$	 is a $C^k\!$  first integral of $\dot{\mathbf{x}}=\mathbf{G}(\mathbf{x})$, then the function $H(\mathbf{x}):=(I\circ\bm{\Phi})(\mathbf{x})$  is a $C^k\!$ first integral of  system \eqref{main-system}.

\item  If $I_j\colon (\R^n,\mathbf{0})\longrightarrow \R,$  $j=1,\ldots,m\le n-1,$	 are $m$ functionally independent  $C^k\!$ first integrals of $\dot{\mathbf{x}}=\mathbf{G}(\mathbf{x})$, then the functions $H_j(\mathbf{x}):=(I_j\circ\bm{\Phi})(\mathbf{x}),$ $j=1,\ldots,m,$ are  functionally independent $C^k\!$ first integrals of  system \eqref{main-system}.

\item If $\mathbf{G}(\mathbf{0})\neq \mathbf{0}$,
then  system~\eqref{main-system} is $C^k\!$ completely integrable in $(\R^n,\mathbf{0})$.
\end{enumerate}	
\end{theorem}

Notice that if $\mathbf{u}=\bm{\Phi}(\mathbf{x})$ preserves dimension and system 
$\dot{\mathbf{u}}=\mathbf{G}(\mathbf{u})$ is completely integrable in $(\R^n,\mathbf{0})$, then Theorem~\ref{th:main} ensures that system~\eqref{main-system} is also completely integrable in $(\R^n,\mathbf{0})$. In such a case, the phase portrait of system~\eqref{main-system} can be obtained from the phase portrait of $\dot{\mathbf{u}}=\mathbf{G}(\mathbf{u})$.
As we will see, in many cases we will take $R=1.$ It is clear that this function simply corresponds to a reparameterization of the trajectories of the system and it has no relevance when people is interested on integrability matters.
For brevity, under the hypothesis of Theorem \ref{th:main}, we call the system $\dot{\mathbf{u}}=\mathbf{G}(\mathbf{u})$ {\it a reduced differential system for  system~\eqref{main-system} through $\bm{\Phi}$}.

To the best of our knowledge, our theorem extends known results for $n=2$ (see Section~\ref{se:plano}) and provides a new mechanism in the search of (complete) integrable differential systems in any dimension.

Next remark presents two examples that show the importance of the hypothesis on $\mathbf{\Phi}$ in Theorem~\ref{th:main}.

\begin{remark} (i)	Consider $\mathbf{x} = (x, y)$,
$\mathbf{u} = (u, v)$, $R=1,$ $\mathbf{G}(\mathbf{u})= (1,2u)^T$ and $\mathbf{\Phi}(\mathbf{x})=(x,-x^2+k),$ for some constant $k\in\R.$ Then $I(\mathbf{u})=u^2+v$ is a first integral of $\dot {\mathbf{u}}=\mathbf{G}(\mathbf{u})$ but $H(\mathbf{x})=I(\mathbf{\Phi}(\mathbf{x}))\equiv k$ it is not a first integral of  $\dot {\mathbf{x}}=\mathbf{F}(\mathbf{x}).$ Obviously, $\mathbf{\Phi}$ does not preserve dimension.
	
(ii) Consider $\mathbf{x} = (x, y, z)$,
$\mathbf{u} = (u, v, w)$,
  $R=1,$ $\mathbf{G}(\mathbf{u})= (0, v, -w)^T$ and $\mathbf{\Phi}(\mathbf{x})=(yz, y ,z).$ Then $I_1(\mathbf{u})=u$ and $I_2(\mathbf{u})=uv$ are functionally independent first integrals of $\dot {\mathbf{u}}=\mathbf{G}(\mathbf{u})$ and	 $H_1(\mathbf{x})=I_1(\mathbf{\Phi}(\mathbf{x}))=yz$ and  $H_2(\mathbf{x})=I_2(\mathbf{\Phi}(\mathbf{x}))=yz$ are also first integrals of  $\dot {\mathbf{x}}=\mathbf{F}(\mathbf{x}),$ but they are not functionally independent because both coincide. Again, $\mathbf{\Phi}$ does not preserve dimension.
	\end{remark}	

Nevertheless, even when  $\mathbf{\Phi}$ does not preserve dimension, the approach of the theorem can sometimes  work to provide complete integrable systems. For instance, when $m=n-1$ first integrals $I_j$ are explicit and functionally independent it may happen that the corresponding $H_j$ are also functionally independent. This situation is illustrated in Example~\ref{Ejemplo-con-diferencial-degenerada-2D} for $n=2$ (see  Section~\ref{se:plano}) and in Example~\ref{Ejemplo-con-diferencial-degenerada} for $n=3$ (see Section~\ref{sec:3D}). In fact, in these examples the determinant of 
	$D\bm{\Phi}(\mathbf{x})$ vanishes identically.

In general, Theorem \ref{th:main} has two immediate applications:  it allows to construct complete integrable differential systems having singularities from differential systems without singularities and it  gives a method to obtain systems in any dimension with a singularity at the origin and with local (complete) integrability.

We also notice that when in system~\eqref{main-system} we consider  $\mathbf{F}=\mathbf{G}$ and 
$\bm{\Phi}$ an involution, then it is said that the system is \emph{orbitally $\bm{\Phi}$-symmetric} and in such a case the system presents strong symmetries that make its study simpler. See \cite{BBT} and references there in.

\medskip
The paper is organized as follows.  Section 2 provides the proof of Theorem~\ref{th:main} as well as  some comments on it. In Section 3, we compare the family of integrable planar systems derived from Theorem~\ref{th:main} with previous results on the integrability of planar analytic differential systems. Additionally, we present several examples of the application of Theorem~\ref{th:main} in the real plane, including families of Liénard, Lotka–Volterra, and quadratic differential systems. Section 4 focuses on examples of the application of Theorem~\ref{th:main} in the three dimensional setting, which include differential systems of  Kolmogorov, Rikitake and Rössler type.
Finally, in Section 5 we present  some  applications in higher dimensions.

\section{Proof of Theorem~\ref{th:main}}

In this section we will prove Theorem~\ref{th:main} and state some remarks around natural extensions.

\begin{proof}[Proof of Theorem \ref{th:main}]
(i) Let $I\colon (\R^n,\mathbf{0})\longrightarrow \R$	 be a $C^k\!$  first integral of $\dot{\mathbf{u}}=\mathbf{G}(\mathbf{u}),$  that is,
\begin{equation}
	\label{Condicion-I_j(u)}
	\nabla I(\mathbf{u})
	\, \mathbf{G}(\mathbf{u}) = 0,
	\quad \mbox{for all}\quad \mathbf{u}\in (\R^n,\mathbf{0}).
\end{equation}
Let us prove that the $C^k\!$ function $H(\mathbf{x}):=(I\circ\bm{\Phi})(\mathbf{x})$ it is a first integral of  system \eqref{main-system}, 
\[
\begin{aligned}
	\dot{H}(\mathbf{x})
	&=
	\nabla H(\mathbf{x})\,R (\mathbf{x})\,\adj(D\bm{\Phi}(\mathbf{x}))\,\mathbf{G}(\bm{\Phi}(\mathbf{x}))\\
	&=\nabla I(\bm{\Phi}(\mathbf{x}))\, 
	D\bm{\Phi}(\mathbf{x})\, R (\mathbf{x})\,
	\adj(D\bm{\Phi}(\mathbf{x}))\,\mathbf{G}(\bm{\Phi}(\mathbf{x})) \\
	&= \nabla I(\bm{\Phi}(\mathbf{x}))\,R (\mathbf{x})\,
	\det (D\bm{\Phi}(\mathbf{x}))\,\mathbf{I}_n\,\mathbf{G}(\bm{\Phi}(\mathbf{x}))\\
	&=R (\mathbf{x})\,\det(D\bm{\Phi}(\mathbf{x}))\,\nabla I(\bm{\Phi}(\mathbf{x}))\, \mathbf{G}(\bm{\Phi}(\mathbf{x}))=0, \quad \mbox{for}\quad  \mathbf{x}\in (\R^n,\mathbf{0}),
	\end{aligned}
\]
where   we  have used that $\nabla I(\bm{\Phi}(\mathbf{x}))\, \mathbf{G}(\bm{\Phi}(\mathbf{x}))=0$, which follows by replacing $\mathbf{u}=\bm{\Phi}(\mathbf{x})$ in equation \eqref{Condicion-I_j(u)}. Hence it only remains to prove that $H$ is not constant on open sets. We already know that this is the case for the function $I,$ because it is a first integral of 
$\dot{\mathbf{u}}=\mathbf{G}(\mathbf{u}).$ Then, since $\mathbf{\Phi}$ preserves dimension this property is translated to $H,$ as we wanted to prove.

(ii)  Assume now that $\dot{\mathbf{x}}=\mathbf{G}(\mathbf{x})$ has  $m\le n-1$   functionally independent 
$C^k\!$ first integrals
$I_1(\mathbf{u}),\ldots,I_{m}(\mathbf{u})$
of the system 
defined in $(\R^n,\mathbf{0})$.   Hence, for each $j \in \{1,\ldots,m\}$ we have
\begin{equation*}
	\nabla I_j(\mathbf{u})
	\, \mathbf{G}(\mathbf{u}) = 0,
	 \quad \mbox{for}\quad \mathbf{u}\in (\R^n,\mathbf{0})
\end{equation*}
and, moreover,
$\nabla I_1(\mathbf{u}),\ldots,\nabla I_{m}(\mathbf{u})$ are linearly independent for almost  any point $\mathbf{u} \in (\R^n,\mathbf{0})$.

We now define, for each $j\in \{1,\ldots,m\}$, the $C^k\!$ function $H_j(\mathbf{x}):=(I_j\circ\bm{\Phi})(\mathbf{x})$.  By item~(i) each $H_j$  is a first integral of system~\eqref{main-system}.

We will prove  by contradiction that all $H_1(\mathbf{x}),\ldots,H_{m}(\mathbf{x})$ are also functionally independent. Assume that the $C^k\!$ first integrals 
$
H_1(\mathbf{x}), 
\ldots, 
H_{m}(\mathbf{x})
$
are functionally dependent in 
$(\R^n,\mathbf{0})$, that is,
there exists an open subset 
$U\subset (\R^n,\mathbf{0})$ such that for each $\mathbf{x} \in U$ there are real constants $\lambda_{1}(\mathbf{x}),\ldots,$ $\lambda_{m}(\mathbf{x})$ not all of them zero, such that 
\[
\lambda_{1}(\mathbf{x}) \nabla H_1(\mathbf{x})+\cdots + 
\lambda_{m}(\mathbf{x}) \nabla H_{m}(\mathbf{x})=\mathbf{0}.
\]
Hence, 
$$
\big(\lambda_{1}(\mathbf{x})
\nabla I_1(\bm{\Phi}(\mathbf{x}))
+ \cdots + 
\lambda_{m}(\mathbf{x})
\nabla I_{m}(\bm{\Phi}(\mathbf{x}))\big)D\bm{\Phi}(\mathbf{x})=\mathbf{0}.
$$
Since $\bm{\Phi}$ preserves dimension, we can assume, without lost of generality, that 
$
\det(D\bm{\Phi}(\mathbf{x}))\neq 0$ for each $\mathbf{x} \in U$. Therefore,
$$
\lambda_{1}(\mathbf{x})
\nabla I_1(\mathbf{u})
+ \cdots + 
\lambda_{m}(\mathbf{x})
\nabla I_{m}(\mathbf{u})=0,
$$
where we have chosen $\mathbf{u}=\bm{\Phi}(\mathbf{x})$
in a non empty open subset of $\bm{\Phi}(U)$. The last equality  contradicts the linear independence of  $\nabla I_1(\mathbf{u}),\ldots,\nabla I_{m}(\mathbf{u})$ 
in   a 
full Lebesgue measure open subset of
$(\R^n, \mathbf{0})$.

(iii) 
When 
$\mathbf{G}(\mathbf{0})\not =\mathbf{0}$,
the flow box theorem 
guarantees that the $C^k\!$ differential system
$
\dot{\mathbf{u}}=\mathbf{G}(\mathbf{u})
$
is $C^k\!$ completely integrable 
in $(\R^n,\mathbf{0})$, that is, 
there are $n-1$ functionally independent 
$C^k\!$ first integrals
$I_1(\mathbf{u}),\ldots,I_{n-1}(\mathbf{u})$
of the system 
defined in $(\R^n,\mathbf{0}).$ Hence the result follows by item~(ii). 
\end{proof}

Next results are straightforward consequences of Theorem~\ref{th:main}.

\begin{remark}\label{Criterio-para-integrabilidad-2}
If differential system $\dot{\mathbf{u}}=\mathbf{G}(\mathbf{u})$
is $C^{\beta}\!$ completely integrable 
in $(\R^n,\mathbf{0})$  and it is a reduced differential system for  system 
\eqref{main-system} through a $C^{\alpha}\!$  function ${\bm{\Phi}\colon
(\R^n,\mathbf{0}) \longrightarrow (\R^n,\mathbf{0}),}$ that preserves dimension, then
 system \eqref{main-system}
is $C^{k}\!$ completely integrable in $(\R^n,\mathbf{0})$, with $k\geq \min\{\alpha,\beta\}$.
\end{remark}

\begin{remark}
\label{Teorema-con-punto-arbitrario}
\rm
We have considered differential systems in $(\R^n,\mathbf{0})$  only for simplicity. Of course, 
we can consider differential systems in $(\R^n,\mathbf{p})$, with $\mathbf{p}\in \R^n$ a singular point. For these systems we can use $C^k\!$ functions $\bm{\Phi}, \mathbf{G} \colon
(\R^n,\mathbf{p}) \longrightarrow (\R^n,\mathbf{q})$. Therefore, we have analogous results to Theorem~\ref{th:main}   if these functions satisfy that 
$\bm{\Phi}(\mathbf{p})=\mathbf{q},$ 
$\mathbf{G}(\mathbf{q})\neq \mathbf{0}$ and the other hypotheses of the theorem. Moreover,  Remark~\ref{Criterio-para-integrabilidad-2} also holds. 
\end{remark}

\begin{remark}
\rm Under the hypotheses of Theorem~\ref{th:main}, if 
$\dot{\mathbf{u}}=\mathbf{G}(\mathbf{u})$ is completely integrable in $\R^n$ and $\bm{\Phi}$ is defined on whole $\R^n$, then system~\eqref{main-system} is completely integrable in $\R^n,$ that is {\it globally completely integrable}.
\end{remark}

In order to guarantee  that a $ C^k $ map $ \bm{\Phi}$ preserves dimension and therefore the application of Theorem~\ref{th:main}, we state the following result.

\begin{lemma}
\label{cond-preserva-dimension}
Let  $ \bm{\Phi} \colon (\mathbb{R}^n, \mathbf{0}) \longrightarrow \mathbb{R}^n$  
be a $ C^k $ map, with $ k \in \mathbb{N} \cup \{ \infty, \omega \}.$ 
If  $\det(D\bm{\Phi}(\mathbf{x}))\neq 0$ in a full Lebesgue measure open subset of $(\mathbb{R}^n, \mathbf{0})$, then $\bm{\Phi}$ preserves dimension.
\end{lemma}
\begin{proof}
Follows from the  Inverse Function Theorem.
\end{proof}

\section{Applications in the plane}\label{se:plano}

In this section, 
we  illustrate the application of our main result restricted to the real plane $\R^2$.  We start by giving a corollary of Theorem~\ref{th:main} which is a version of that result  when $n=2$.  We introduce the notations  
 $\mathbf{x}=(x,y),$  $\mathbf{u}=(u,v),$  $\mathbf{\Phi}=(\varphi,\psi)$ and $\mathbf{G}=(f,g)^{T}.$

\begin{corollary}
	\label{sistema-integrable-planar}
Consider $C^r\!$ planar differential system
\begin{equation}
	\label{main-system-n=2}
\left(\begin{array}{c}
	\dot x\\
	\dot y
\end{array}\right)= R(x,y)
\left(\begin{array}{rr}
	\psi_y(x,y) & -\varphi_y(x,y)\\
	-\psi_x(x,y) & \varphi_x(x,y)
\end{array} \right)
\left(\begin{array}{c}
	f(\varphi(x,y),\psi(x,y))\\
	g(\varphi(x,y),\psi(x,y))
\end{array}\right),
\end{equation}	where $\varphi,\psi, f,$ and $g$ are suitable $C^k\!$ scalar functions in $(\R^2,\mathbf{0})$, with $k\geq r$, and $R$ is a real-valued function  such that  $R(\mathbf{x})\neq 0$ in a full Lebesgue measure open subset of $(\R^2,\mathbf{0}).$  If the map $\mathbf{\Phi}=(\varphi,\psi)$ preserves dimension, then  the system is $C^r\!$ 
	integrable in $(\R^2,\mathbf{0})$ provided that
	$\varphi(\mathbf{0})=\psi(\mathbf{0})=0$ 
	and $f^2(\mathbf{0})+g^2(\mathbf{0})\neq 0$.
\end{corollary}

The above corollary extends a similar and previous result proved in \cite{AlGaRe}, which essentially corresponds to  the above one but fixing $\psi(x,y)= y^p/p$ and where the hypothesis that  $\mathbf{\Phi} (x,y)=(\varphi(x,y), y^p/p)$ preserves dimension is replaced by assuming that $\varphi_x(x,y)\not\equiv0.$ Corollary~\ref{sistema-integrable-planar} also covers similar  results that appear in \cite{Andreev}  and in the references of both papers. 

\begin{remark}
	The differential system~\eqref{main-system-n=2} admits a more compact version by using the notation of 1-forms. Notice that the reduced system  through $\mathbf{\Phi}$ can be written as $f(u,v)\,{\rm d}v-g(u,v)\,{\rm d}u=0,$ while system~\eqref{main-system-n=2} writes as 
$$
f(\varphi(x,y),\psi(x,y)) \,{\rm d}(\psi(x,y))- f(\varphi(x,y),\psi(x,y)) \,{\rm d} (\varphi(x,y))=0,
$$
where notice that we have omitted the multiplicative factor $R(x,y).$	 Indeed, this expression is the pull-back of the initial 1-form:
$$
\bm{\Phi}^{*}(f(u,v)\,{\rm d}v-g(u,v)\,{\rm d}u)=0.
$$
\end{remark}

Since in the planar case, the weak and strong versions of the integrability problem coincide,   we will use, for simplicity, the adjective integrable instead of  completely integrable. Next we list several integrable examples in the plane. In the first seven examples the functions $\bm{\Phi}$ that we will use preserve dimension, which can be proved easily using Lemma~\ref{cond-preserva-dimension}.

\begin{example}
\rm
\label{example-Kolmogorov-2D}
By using $R=1$ and the polynomial functions
$$
\bm{\Phi}(\mathbf{x})=
\bigg(xy, x+\frac{y^2}{2}\bigg)
\quad
\mbox{and} 
\quad  
\mathbf{G}(\mathbf{u})=
\big(u,1-u\big)^{T},
$$
system \eqref{main-system} becomes
\begin{equation*}
\begin{array}{l}
\dot{x}=x(-1+xy+y^2),\\[0.4pc]
\dot{y}=y(1-x-xy),
\end{array}
\end{equation*}
which is a  Kolmogorov planar system. 
Since $\bm{\Phi}(\mathbf{0})=\mathbf{0}$
and
$\mathbf{G}(\mathbf{0})\neq \mathbf{0}$,  this  system is $C^{\omega}\!$  integrable in $(\R^2,\mathbf{0})$ by Theorem~\ref{th:main}. 
Moreover, since system 
$\dot{\mathbf{u}}=\mathbf{G}(\mathbf{u})$ has the  first integral $I(\mathbf{u})=ue^{-u-v}$, our Kolmogorov system  is   $C^{\omega}\!$  integrable in whole real plane, with  the analytic Darboux first integral
$
H(\mathbf{x})=xy\exp\big(-x-xy-y^2/2\big)
$. 
Notice that the reduced system is free of finite singularities, while the Kolmogorov system has two singularities on the plane: a saddle at $\mathbf{0}$ and a center at a point in the first quadrant.
\end{example}

\begin{example}
\rm
By taking $R=1$  and the polynomial functions
$$
\bm{\Phi}(\mathbf{x})=
\bigg(\frac{x^2}{2}+\frac{b y^2}{2}, \frac{y^2}{2}\bigg)
\quad
\mbox{and} 
\quad  
\mathbf{G}(\mathbf{u})=
\big(1,-2u\big)^{T},
$$
where $b\in \R$ is a parameter, system \eqref{main-system} becomes
\begin{equation*}
\begin{array}{l}
\dot{x}=y+bx^2y+b^2y^3,\\[0.4pc]
\dot{y}=-x^3-bxy^2,
\end{array}
\end{equation*}
which is a planar system with a nilpotent singularity at the origin. 
Since $\bm{\Phi}(\mathbf{0})=\mathbf{0}$
and
$\mathbf{G}(\mathbf{0})\neq \mathbf{0}$,  this  system is $C^{\omega}\!$  integrable in $(\R^2,\mathbf{0})$ by Theorem~\ref{th:main}. 
Furthermore, since system 
$\dot{\mathbf{u}}=\mathbf{G}(\mathbf{u})$ has the  first integral $I(\mathbf{u})=v+u^2$, our system  is   polynomially  integrable in whole real plane, with  the polynomial first integral
$
H(\mathbf{x})=y^2/2+(x^2/2+by^2/2)^2,
$
which has a minimum at the origin. Since the system is hamiltonian, the origin is a center.
\end{example}

\begin{example} 
\rm
For each $k\in \N$, set $a=\sqrt{2}+k$. The linear differential system
$\dot{\mathbf{u}}=\mathbf{G}_k(\mathbf{u})$ 
given by
\begin{equation}
\label{Sist-lineal-con-integral-primera-Ck}
\begin{array}{l}
\dot{u}=(a+3)u+2av,\\[0.4pc]
\dot{v}=-2au-(4a-3)v,
\end{array}
\end{equation}
has a saddle at the origin and the first quadrant is foliated by orbits of a hyperbolic sector of the saddle.
By using $R=1$ and the polynomial function
$
\bm{\Phi}(\mathbf{x})=
\big(x^2, y^2\big),
$
system~\eqref{main-system} writes as
\begin{equation}
\label{Sist-con-centro-degenerado-de-clase-Ck}
\begin{array}{l}
\dot{x}=2(a+3)x^2y+4ay^3,\\[0.4pc]
\dot{y}=-4ax^3-2(4a-3)xy^2,
\end{array}
\end{equation}
which has a unique degenerated singularity at  $\mathbf{0}.$

The form of $\bm{\Phi}(\mathbf{x})$ implies 
that the orbits of  system~\eqref{Sist-lineal-con-integral-primera-Ck} 
contained in the first quadrant come from  closed curves around the origin in the $xy$-plane. Hence, system~\eqref{Sist-con-centro-degenerado-de-clase-Ck} has a degenerated global center at $\mathbf{0}$. Furthermore,
system~\eqref{Sist-lineal-con-integral-primera-Ck}
is $C^k\!$ integrable in $\R^2$, with  $C^k\!$ first integral $I(\mathbf{u})=(u+2v)(2u+v)^{a-1}$. Since $\bm{\Phi}(\mathbf{x})$ is defined on $\R^2$, Remark~\ref{Criterio-para-integrabilidad-2} implies that system
\eqref{Sist-con-centro-degenerado-de-clase-Ck}
is $C^k\!$ integrable in $\R^2$, with  $C^k\!$ first integral 
$$
H(\mathbf{x})=(x^2+2y^2)(2x^2+y^2)^{a-1}.
$$ 
\end{example}


\begin{example}
\rm
The reversible planar quadratic differential systems with a center at the origin write as
\begin{equation*}
\begin{array}{l}
\dot{x}=y+\alpha x y,\\[0.4pc]
\dot{y}=-x+\beta x^2+\gamma y^2,
\end{array}
\end{equation*}
where $\alpha,\beta,\gamma\in \R$,  and are also called  Loud systems. By Poincaré linearizability Theorem it is already  known that they have an analytic first integral at $(\R^2,\mathbf{0}).$  Let us prove this fact as a consequence of Theorem~\ref{th:main}. They can be written in the form~\eqref{main-system} with $R=1,$
$$
\bm{\Phi}(\mathbf{x})=
\bigg(x, \frac{y^2}{2}\bigg)
\quad
\mbox{and} 
\quad  
\mathbf{G}(\mathbf{u})=
\big(1+\alpha u,-u+\beta u^2+2\gamma v\big)^{T}.
$$
Since $\bm{\Phi}(\mathbf{0})=\mathbf{0}$
and
$\mathbf{G}(\mathbf{0})\neq \mathbf{0}$, the Loud system is locally $C^{\omega}\!$ integrable by Theorem~\ref{th:main}. 
Observe that the reduced system is also quadratic. 
But,  if $\beta (\gamma-\alpha) \neq 0$, then 
we can get a reduced system that is linear. 
Indeed, in such a case the Loud system can be written in the form \eqref{main-system}  with $R=1,$ 
$$
\bm{\Phi}(\mathbf{x})=\bigg({x}^{2}+\frac { \left( \gamma-\alpha \right) {y}^{2}}{\beta},\left( \gamma-\alpha \right) x\bigg)
$$
and
$$ 
\mathbf{G}(\mathbf{u})=\bigg(-{\frac {\beta\,\gamma\,u}{ \left( \gamma-\alpha \right) ^{2}}}+{
\frac { \left( \gamma-\alpha-\beta \right) v}{ \left( \gamma-\alpha
 \right) ^{3}}},
-\frac{1}{2}\,{\frac {\beta\,\alpha\,v}{ \left( \gamma-\alpha \right) ^{2}}}
-\frac{1}{2}\,{\frac {\beta}{\gamma-\alpha}}
\bigg)^{T}.
$$
\end{example}

\begin{example}
\rm
The analytic Li\'enard system 
$$
\begin{array}{l}
\dot{x}=y,\\[0.2pc]
\dot{y}=-xg(x^2)-xyf(x^2),  
\end{array}
$$
where $g$ and $f$ are $C^{\omega}\!$ functions in $(\R, 0)$ and $g(0)\neq 0$, can be written in the form \eqref{main-system} with $R=1,$
$$
\bm{\Phi}(\mathbf{x})=
\big(x^2, y\big)
\quad
\mbox{and} 
\quad  
\mathbf{G}(\mathbf{u})=
\bigg(v,-\frac{g(u)+vf(u)}{2}\bigg)^{T}.
$$
Since $\bm{\Phi}(\mathbf{0})=\mathbf{0}$
and
$\mathbf{G}(\mathbf{0})=(0,-g(0)/2)\neq \mathbf{0}$,  this Li\'enard system is locally $C^{\omega}\!$ integrable by Theorem~\ref{th:main}.
\end{example}

\begin{example}
\rm
The  quadratic differential system 
$$
\begin{array}{l}
\dot{x}=x(ax+by+c),\\[0.3pc]
\dot{y}=y\big(\frac{ab}{B}x+By-c\big),  
\end{array}
$$
where $a,b\in \R$ and $c,B\in \R^{*}:=\R\backslash \{0\}$, is a ($4$-parametric) subcase of the classical Lotka--Volterra systems. It can be written in the form \eqref{main-system}
with  $R=1,$ 
$$
\bm{\Phi}(\mathbf{x})=
\Big(xy, y-\frac{a}{B}x\Big)
\quad
\mbox{and} 
\quad  
\mathbf{G}(\mathbf{u})=
\big((b+B)u,Bv-c\big)^{T}.
$$
Since $\bm{\Phi}(\mathbf{0})=\mathbf{0}$
and
$\mathbf{G}(\mathbf{0})=(0,-c)\neq \mathbf{0}$,  this particular Lotka--Volterra system is locally $C^{\omega}$ integrable by Theorem~\ref{th:main}. Moreover,  the reduced system 
$\dot{\mathbf{u}}=\mathbf{G}(\mathbf{u})$ 
can be solved easily, whence we get that it has 
$$
I(\mathbf{u})=u\big(Bv-c\big)^{-\frac{b+B}{B}}
$$ 
as a $C^{\omega}\!$ first integral in $(\R^2, \mathbf{0})$. Thus, 
$$
H(\mathbf{x})=xy\big(By-ax-c\big)^{-\frac{b+B}{B}}
$$
is a $C^{\omega}\!$ first integral in 
$(\R^2, \mathbf{0})$ of our particular 
Lotka--Volterra system.  
Furthermore,  if $b=-(p+q)B/q$, with $p,q\in \N$, 
then the system is globally polynomially 
integrable, with the polynomial first integral
$
H_{p,q}(\mathbf{x})=(xy)^q(By-ax-c)^p,
$
cf. \cite{LlVa2007}.
\end{example}

\begin{example}
\rm
The planar Lotka--Volterra system 
$$
\begin{array}{l}
\dot{x}=x(ax+by+c),\\[0.3pc]
\dot{y}=y\Big(\frac{a(B-b)}{b}x+\frac{B}{2}y+\frac{Bc}{b}\Big),  
\end{array}
$$
where $a,c\in \R$, $b,B\in \R^{*}$ and $bB<0$, can be written in the form \eqref{main-system}
with  $R=1,$
$$
\bm{\Phi}(\mathbf{x})=
\bigg(x, cy+axy+\frac{b}{2} y^2\bigg)
\quad
\mbox{and} 
\quad  
\mathbf{G}(\mathbf{u})=
\bigg(u,\frac{B}{b}v\bigg)^{T}.
$$
Thus, 
$\dot{\mathbf{u}}=\mathbf{G}(\mathbf{u})$ 
is  a reduced differential system for 
this particular Lotka--Volterra system and it can be solved easily, whence we get that either
$I(\mathbf{u})=uv^{-b/B}$  or $I(\mathbf{u})=vu^{-B/b}$ 
is a $C^l\!$ first integral in $\R^2$ 
of the reduced system, 
with $l=[-b/B]\in \N$ or $l=[-B/b]\in \N$, respectively, where $[r]$ stands for the integer part of the real number $r$. 
In addition, if  $-b/B$ is a positive rational number $p/q$, then 
$I(\mathbf{u})=u^{q}v^{p}$ 
is a polynomial first integral in $\R^2$ 
of the reduced system. 

Since $\bm{\Phi}(\mathbf{x})$ is polynomial, $\bm{\Phi}(\mathbf{0})=\mathbf{0}$
and system 
$\dot{\mathbf{u}}=\mathbf{G}(\mathbf{u})$ 
is  polynomially or $C^l\!$  integrable in $\R^2$,   
the Lotka--Volterra system is   polynomially or $C^l\!$ integrable  in $\R^2$ by Remark~\ref{Criterio-para-integrabilidad-2}, with polynomial first integral 
$$
H(\mathbf{x})=x^q\bigg(cy+axy+\frac{b}{2} y^2\bigg)^{p},
$$
if $-b/B=p/q\in \mathbb{Q}^{+}$, and $C^l\!$ first integral either
$$
H(\mathbf{x})=x\bigg(cy+axy+\frac{b}{2} y^2\bigg)^{-b/B}
\quad  \mbox{or} \quad 
H(\mathbf{x})=\bigg(cy+axy+\frac{b}{2} y^2\bigg)x^{-B/b}
$$
if $-b/B\in \R^{+}\setminus\mathbb{Q}^{+}$ and either $l=[-b/B]\in \N$ or $l=[-B/b]\in \N,$
cf. \cite{LlVa2007}.
\end{example}

\medskip
We finish this section with an example where the function $\mathbf{\Phi}$ does not preserve dimension but the approach developed also allows to prove (complete) integrability. 

\begin{example}
	\label{Ejemplo-con-diferencial-degenerada-2D}
	\rm 
	By using $R=1$ and the  functions
	$$
	\bm{\Phi}(\mathbf{x})=
	(x,s(x))
	\quad
	\mbox{and} 
	\quad  
	\mathbf{G}(\mathbf{u})=
	(1,2u)^T,
	$$
where $s(x)$ is a $C^2\!$ function satisfying that $x^2+s(x)$ is not constant on open intervals,	system \eqref{main-system} becomes
	\begin{equation*}
		\begin{array}{l}
			\dot{x}= 0,\\[0.4pc]
			\dot{y}=2x-s'(x).
		\end{array}
	\end{equation*}
	Here
	$\bm{\Phi}(\mathbf{0})=\mathbf{0}$ and 
	$\mathbf{G}(\mathbf{0})\ne\mathbf{0}$ but we can not apply  Theorem~\ref{th:main} because $\mathbf{\Phi}$ does not preserve dimension.  	Notice that  $\det (D\bm{\Phi}(\mathbf{x}))\equiv 0.$  Nevertheless, since  $I(\mathbf{u})=u^2+v$ is a first integral of the reduced system $\dot {\mathbf{u}}=\mathbf{G}(\mathbf{u})$
	 and 
	$H(\mathbf{x})=I(\mathbf{\Phi}(\mathbf{x}))= I(x,s(x))= x^2+s(x)$  is a first integral of  the system. Clearly, there is a more natural first integral $J(\mathbf{x})=x.$
	\end{example}

\section{Applications in the space}\label{sec:3D}

In this section, 
we  illustrate the application of our main results restricted to the three dimensional real space $\R^3$. As in the 2-dimensional case we state a corollary of Theorem~\ref{th:main} for this case. We introduce the notation
$\mathbf{x}=(x,y,z),$  $\mathbf{u}=(u,v,w),$  $\mathbf{\Phi}=(\varphi,\psi,\eta)$ and $\mathbf{G}=(f,g,h)^{T}.$

\begin{corollary}
	\label{sistema-integrable-space}
	Consider $C^r\!$ differential system in the space\[
\left(\begin{array}{c}
	\dot x\\
	\dot y\\
	\dot z
\end{array}\right)\!=\! R \!
\left(\begin{array}{rrr}
	\psi_y\eta_z-\psi_z\eta_y & \varphi_z\eta_y-\varphi_y\eta_z & \varphi_y\psi_z-\varphi_z\psi_y\\
	\psi_z\eta_x-\psi_x\eta_z	&\varphi_x\eta_z-\varphi_z\eta_x & \varphi_z\psi_x-\varphi_x\psi_z \\
	\psi_x\eta_y-\psi_y\eta_x	& \varphi_y\eta_x-\varphi_x\eta_y & \varphi_x\psi_y-\varphi_y\psi_x
\end{array} \right)\!\!
\left(\begin{array}{c}
	f(\varphi,\psi,\eta)\\
	g(\varphi,\psi,\eta)\\
	h(\varphi,\psi,\eta)
\end{array}\right)\!,
\]	
where $\varphi,\psi,\eta, f, g, $ and $h$ are suitable $C^k\!$ scalar functions in $(\R^3,\mathbf{0})$, with $k\geq r$, and $R$ is a real-valued function   such that  $R\neq 0$ in a full Lebesgue measure open subset of $(\R^3,\mathbf{0})$ (we have omitted the dependence on $(x,y,z)$ for simplicity).  If the map $\mathbf{\Phi}=(\varphi,\psi,\eta)$ preserves dimension, then  the system is $C^r\!$ 
	integrable in $(\R^3,\mathbf{0})$ provided that
	$\varphi(\mathbf{0})=\psi(\mathbf{0})=\eta(\mathbf{0})=0$ 
	and $f^2(\mathbf{0})+g^2(\mathbf{0})+h^2(\mathbf{0})\neq 0$.
\end{corollary}

We continue presenting several examples of application of Theorem~\ref{th:main} when $n=3.$

\begin{example}\rm {\bf A Kolmogorov system.}  Consider 
$$
R=1,\quad \bm{\Phi}(\mathbf{x})=
\big(xyz,
\psi(x,y,z),
z\big)
\quad
\mbox{and} 
\quad  
\mathbf{G}(\mathbf{u})=
\big(0,0,1\big)^T,
$$	
where $\psi(x,y,z)$ is a $C^r\!$ function in  $(\R^3, \mathbf{0})$, with $\psi(\mathbf{0})=0$ and such that $\bm{\Phi}$ preserves dimension. The  system \eqref{main-system}  in  Theorem~\ref{th:main} becomes
\begin{equation*}
\begin{array}{l}
\dot{x}=x\big(z\psi_z(x,y,z)-y\psi_y(x,y,z)\big),\\[0.4pc]
\dot{y}=y\big(x\psi_x(x,y,z)-z\psi_z(x,y,z)\big),\\[0.4pc]
\dot{z}=z\big(y\psi_y(x,y,z)-x\psi_x(x,y,z)\big),
\end{array}
\end{equation*}
which is a Kolmogorov system in 
$(\R^3, \mathbf{0})$.
Since 
$\bm{\Phi}(\mathbf{0})=\mathbf{0}$ and 
$\mathbf{G}(\mathbf{0})\ne \mathbf{0}$, this Kolmogorov system is $C^{r}\!$ completely integrable in $(\R^2, \mathbf{0})$ by Theorem~\ref{th:main}. 
Furthermore, if $\psi(x,y,z)$ is a polynomial function, then the system is  polynomially completely integrable in $\R^3$, because $\bm{\Phi}(\mathbf{x})$ becomes polynomial and the system 
$\dot{\mathbf{u}}=\mathbf{G}(\mathbf{u})$ is polynomially completely integrable in $\R^3$ with trivial first integrals $I_1(\mathbf{u})=u$ and $I_2(\mathbf{u)}=v.$ 
\end{example}

In the next four examples the functions $\bm{\Phi}$ that we will use preserve dimension according to Lemma~\ref{cond-preserva-dimension}.

\begin{example} {\bf A second Kolmogorov system.}
\rm
Consider the linear differential system 
$\dot{\mathbf{u}}=\mathbf{G}(\mathbf{u})$ 
given by
\begin{equation}
\label{Sist-lineal-compl-int-R3}
\begin{array}{l}
\dot{u}=u,\\[0.2pc]
\dot{v}=-v,\\[0.2pc]
\dot{w}=-w.
\end{array}
\end{equation}
By taking $R=1$ and  the polynomial function
$
\bm{\Phi}(\mathbf{x})=
\big(yz, xz+x^2,xz-y\big),
$
the differential system~\eqref{main-system} is
\begin{equation}
\label{Sist-compl-int-R3}
\begin{array}{l}
\dot{x}=x(xy+3yz+x^2z),\\[0.4pc]
\dot{y}=y(2xy+yz-3x^2z),\\[0.4pc]
\dot{z}=z(-4xy-2yz+x^2z),
\end{array}
\end{equation}
which is a Kolmogorov system.
System 
\eqref{Sist-lineal-compl-int-R3}
is polynomially completely integrable in $\R^3$, with  polynomial first integrals $I_1(\mathbf{u})=uv$ and $I_2(\mathbf{u})=uw$. Since $\bm{\Phi}(\mathbf{x})$ is polynomial on $\R^3$, Theorem~\ref{th:main} implies that system
\eqref{Sist-compl-int-R3}
is polynomially completely integrable in $\R^3$, with functionally independent polynomial first integrals
$$
H_1(\mathbf{x})=yz(xz+x^2)
\quad
\mbox{and}
\quad
H_2(\mathbf{x})=yz(xz-y).
$$ 
\end{example}

\begin{example}
\rm
{\bf A complete integrable R\"ossler system.} 
The  polynomial  differential system
\begin{equation}
\label{Rossler-system}
\begin{array}{l}
\dot{x}=-y-z,\\[0.4pc]
\dot{y}=x,\\[0.4pc]
\dot{z}=xz,
\end{array}
\end{equation}
is a particular case of the differential
system constructed by R\"ossler \cite{Ros1976}. It can be written in the form \eqref{main-system}
with the analytic functions $R=1,$
$$
\bm{\Phi}(\mathbf{x})=
\Big(\frac{x^2+y^2}{2}+z,
y,
z\Big)
\quad
\mbox{and}
\quad 
\mathbf{G}(\mathbf{u})=
\big(0,1,w\big)^T.
$$
Since 
$\Phi(\mathbf{0})=\mathbf{0}$
and
$\mathbf{G}(\mathbf{0})\neq \mathbf{0}$,  system 
\eqref{Rossler-system} is analytically completely integrable by Theorem~\ref{th:main}. Moreover, it is easy to show that system 
$\dot{\mathbf{u}}=\mathbf{G}(\mathbf{u})$ has the functionally independent first integrals $I_1(\mathbf{u})=u$ and $I_2(\mathbf{u})=we^{-v}$. Hence, R\"ossler system \eqref{Rossler-system}   is also Darboux completely integrable, with  independent Darboux first integrals
$$
H_1(\mathbf{x})=\frac{x^2+y^2}{2}+z
\qquad
\mbox{and} 
\qquad 
H_2(\mathbf{x})=ze^{-y},
$$
cf. \cite{LlZh2002, CMF2016}.
\end{example}

\begin{example}
\rm
{\bf A complete integrable Rikitake system.} 
The  polynomial 
differential system
\begin{equation}
\label{Rikitake-system}
\begin{array}{l}
\dot{x}=yz,\\[0.4pc]
\dot{y}=xz,\\[0.4pc]
\dot{z}=1-xy,
\end{array}
\end{equation}
is a particular case of the Rikitake system \cite{Rik1958}. It can be written in the form~\eqref{main-system}
with the analytic functions $R=1,$
$$
\bm{\Phi}(\mathbf{x})=
\Big(\frac{x+y}{2},
y-x,
x^2+z^2\Big)
\quad
\mbox{and} 
\quad  
\mathbf{G}(\mathbf{u})=
\Big(\frac{u}{2},-\frac{v}{2},1\Big)^T.
$$
The point $\mathbf{p}=(1,1,0)$ is a singular point of the system. Since 
$\Phi(\mathbf{p})=(1,0,1)$
and
$\mathbf{G}(1,0,1)\neq \mathbf{0}$,  system 
\eqref{Rikitake-system} is analytically completely integrable by Remark~\ref{Teorema-con-punto-arbitrario} and Theorem~\ref{th:main}. Moreover, it is easy to show that system 
$\dot{\mathbf{u}}=\mathbf{G}(\mathbf{u})$ has the functionally independent  first integrals $I_1(\mathbf{u})=uv$ and $I_2(\mathbf{u})=v^2e^{w}$. Hence, Rikitake system \eqref{Rikitake-system}  is also Darboux completely integrable, with  independent Darboux first integrals
$$
H_1(\mathbf{x})=\frac{y^2-x^2}{2}
\qquad
\mbox{and} 
\qquad 
H_2(\mathbf{x})=(y-x)^2\,e^{x^2+z^2},
$$
cf. \cite{LlZh2000,LlMe2009}.
\end{example}

\begin{example}
\rm
{\bf A complete integrable 
Lotka--Volterra system.} 
The  polynomial 
dif\-fer\-en\-tial system
\begin{equation}
\label{LV-system-1}
\begin{array}{l}
\dot{x}=x(z-cy),\\[0.4pc]
\dot{y}=y(x-az),\\[0.4pc]
\dot{z}=z\Big(y-\dfrac{x}{ac}\Big),
\end{array}
\end{equation}
with $a,c\in \R^{+}$, is part of the so-called $(a,b,c)$  Lotka--Volterra systems \cite{Lab1996}.  

If $c \geq 1$,  then system \eqref{LV-system-1} can be written in the form \eqref{main-system}, 
with the $C^1\!$ functions $R(\mathbf{x})=(1/c)z^{1-c},$ 
$$
\bm{\Phi}(\mathbf{x})=
\big(acz+cy+x,
y,
 xz^{c}\big),
\ \
\mbox{and}
\ \
\mathbf{G}(\mathbf{u})=
\Big(0,v,-\frac{w}{a}\Big)^T.
$$
If $c<1$, then system \eqref{LV-system-1} can be written in the form \eqref{main-system}, 
 with the $C^1\!$ functions $R(\mathbf{x})=x^{1-1/c},$  
 $$
\bm{\Phi}(\mathbf{x})=
\big(acz+cy+x,
y,
 x^{1/c}z\big),
\ \
\mbox{and}
\ \
\mathbf{G}(\mathbf{u})=
\Big(0,v,-\frac{w}{ac}\Big)^T.
$$

In both cases, $c\geq 1$ and $c<1$, 
 the reduced differential system 
can be solved easily, whence we get that 
${I_1(\mathbf{u})=u}$ is a polynomial 
first integral of 
$\dot{\mathbf{u}}=\mathbf{G}(\mathbf{u})$. 
Moreover, on one hand,
for case $c\geq 1$ we obtain that 
if $a\geq 1$, then  
${I_2(\mathbf{u})=vw^{a}}$ 
is a $C^{\beta}\!$ first integral 
in $\R^3$ of the reduced system, 
with  $\beta=[a]\in \N$, and 
if $a<1$, then 
${I_2(\mathbf{u})=v^{1/a}w}$ 
is a $C^{\beta}\!$ first integral 
in $\R^3$ of the reduced system, 
with $\beta=[1/a]\in \N$. On the other hand,
for case $c< 1$  it follows that if $a\geq 1$, then  
${I_2(\mathbf{u})=v^{1/c}w^a}$ 
is a $C^{\beta}\!$ first integral in $\R^3$ 
of the reduced system, 
with  $\beta=\min\{[a],[1/c]\}\in \N$, and
if $a<1$, then
${I_2(\mathbf{u})=v^{1/(ac)}w}$
is a $C^{\beta}\!$ first integral in $\R^3$ 
of the reduced system, 
with $\beta=[1/(ac)]\in \N$.


Since $\bm{\Phi}(\mathbf{0})=\mathbf{0}$, $\bm{\Phi}(\mathbf{x})$ is a polynomial or $C^{\alpha}\!$ function, with $\alpha=[c]\in \N$ if $c\geq 1$ and $\alpha=[1/c]\in \N$ if $c<1$, 
and system 
$\dot{\mathbf{u}}=\mathbf{G}(\mathbf{u})$ 
is  polynomially or $C^{\beta}\!$ completely integrable in $\R^3$,   
the Lotka--Volterra system is  polynomially or
$C^k\!$  completely integrable  in $\R^3$ by Remark~\ref{Criterio-para-integrabilidad-2}, with $k\geq 1$. 
More precisely, 
the Lotka--Volterra system always has the polynomial first integral 
$$
H_1(\mathbf{x})=x+cy+acz.
$$
If $a=p/q\in \mathbb{Q}^{+}$ and $c=p_1/q_1\in \mathbb{Q}^{+}$, then the Lotka--Volterra system is  polynomially  completely integrable in $\R^3$, with the additional independent polynomial first integral 
$$
H_2(\mathbf{x})=x^{pq_1}y^{qq_1}z^{pp_1}.
$$

If $c\not \in \mathbb{Q}^{+}$ and $c>1$, then the Lotka--Volterra system is  $C^k\!$ completely integrable  in $\R^3$, with the additional functionally independent $C^k\!$ first integral either
$$
H_2(\mathbf{x})=x^ayz^{ac}
\ \
\mbox{or}
\ \
H_2(\mathbf{x})=xy^{1/a}z^{c},
$$
where $k\geq\min\big\{[a],[ac]\big\}$ if $a>1$ or $k\geq\min\big\{[1/a],[c]\big\}$ if $a<1$.

If $c\not \in \mathbb{Q}^{+}$ and $c<1$, then the Lotka--Volterra system is  $C^k\!$ completely integrable  in $\R^3$, with the additional functionally independent $C^k\!$ first integral either
$$
H_2(\mathbf{x})=x^{a/c}y^{1/c}z^{a}
\ \
\mbox{or}
\ \
H_2(\mathbf{x})=x^{1/c}y^{1/(ac)}z,
$$
where $k\geq\min\big\{[1/c],[a]\big\}$ if $a>1$ or $k\geq\min\big\{[1/c],[1/(ac)]\big\}$ if $a<1$,
cf. \cite{Lab1996, CaLl2000}.
\end{example}

We conclude this section with an example where the function $\mathbf{\Phi}$ does not preserve dimension but the approach developed also allows to prove complete integrability. 

\begin{example}
	\label{Ejemplo-con-diferencial-degenerada}
	\rm
	By using $R=1$ and the polynomial functions
	$$
	\bm{\Phi}(\mathbf{x})=
	\big(xz-y^2,
	yz-y^2,
	xz-yz\big)
	\quad
	\mbox{and} 
	\quad  
	\mathbf{G}(\mathbf{u})=
	\big(0,0,1\big)^T
	$$
	system \eqref{main-system} becomes
	\begin{equation*}
		\begin{array}{l}
			\dot{x}= 2xy-xz-2y^2,\\[0.4pc]
			\dot{y}=-yz,\\[0.4pc]
			\dot{z}=-2yz+z^2.
		\end{array}
	\end{equation*}
	Here
	$\bm{\Phi}(\mathbf{0})=\mathbf{0}$ and 
	$\mathbf{G}(\mathbf{0})\ne\mathbf{0}$ but we can not apply  Theorem~\ref{th:main} because $\mathbf{\Phi}=(\varphi,\psi,\eta)$ does not preserve dimension. In fact, its components satisfy $\varphi-\psi-\eta=0.$ 	Notice that  $\det (D\bm{\Phi}(\mathbf{x}))\equiv 0.$  Nevertheless the system 
	is polynomially completely integrable in $(\R^3, \mathbf{0}).$ The reason is that 
	$\bm{\Phi}(\mathbf{x})$ is polynomial and its reduced system 
	is polynomially completely integrable, with first integrals $I_1(\mathbf{u})=u$ and $I_2(\mathbf{u})=v.$  Hence their transformed first integrals 
	$H_1(\mathbf{x})= xz-y^2$ and 
	$H_2(\mathbf{x})= yz-y^2$ are also functionally independent.
	\end{example}

\section{Applications in higher dimensions}

In this section, 
we  illustrate the application of our main result in higher dimensions. In particular, an example for arbitrary dimension is presented. Again the functions $\bm{\Phi}$ in next examples preserve dimension according to Lemma~\ref{cond-preserva-dimension}.

\begin{example}
\rm
{\bf Some 4D complete integrable 
Kolmogorov systems}. For each $l\in \N$, consider the polynomial Kolmogorov system 
\begin{equation}
\label{LV-system-2}
\begin{array}{l}
\dot{x}_1=x_1(x_2^l-x_4^l),\\[0.4pc]
\dot{x}_2=x_2(x_3^l-x_1^l),\\[0.4pc]
\dot{x}_3=x_3(x_4^l-x_2^l),\\[0.4pc]
\dot{x}_4=x_4(x_1^l-x_3^l),
\end{array}
\end{equation}
which for $l=1$ is a Lotka--Volterra system. This differential system
can be written in the form \eqref{main-system}
with the analytic functions $R=1,$
$$
\bm{\Phi}(\mathbf{x})=
\big(x_1,x_1x_3,x_2x_4,x_1^l+x_2^l+x_3^l+x_4^l
\big)
$$
and
$
\mathbf{G}(\mathbf{u})=
({1}/{l},0,0,0)^T.
$
Since 
$\Phi(\mathbf{0})=\mathbf{0}$
and
$\mathbf{G}(\mathbf{0})\neq \mathbf{0}$,  system \eqref{LV-system-2} is analytically completely integrable by Theorem~\ref{th:main}. Moreover, the 
differential system 
$\dot{\mathbf{u}}=\mathbf{G}(\mathbf{u})$ is polynomially completely integrable, with polynomial first integrals 
$
I_i(\mathbf{u})=u_{i+1}$,   for $i=1,2,3$.
Hence, system \eqref{LV-system-2} is polynomially completely integrable, with independent polynomial first integrals
$$ H_1(\mathbf{x})=x_1x_3,\quad 
H_2(\mathbf{x})=x_2x_4
\quad
\mbox{and}
\quad
H_3(\mathbf{x})=\sum_{i=1}^4x_i^l.
$$ 
\end{example}

\begin{example}
\rm
{\bf Some 4D complete integrable 
nilpotent systems}. Consider the polynomial differential systems 
\begin{equation}
\label{nilpotente}
\begin{array}{l}
\dot{x}_1=P_1\big(x_2+A_1(x_1)\big),\\[0.4pc]
\dot{x}_2=P_2\big(x_3+\frac{1}{d_2}A_2(x_1)\big)-A_1'(x_1)P_1\big(x_2+A_1(x_1)\big),\\[0.4pc]
\dot{x}_3=P_3\big(x_4+\frac{1}{d_3}A_3(x_1)\big)-\frac{A_2'(x_1)}{d_2}P_1\big(x_2+A_1(x_1)\big),\\[0.4pc]
\dot{x}_4=-\frac{A_3'(x_1)}{d_3}P_1\big(x_2+A_1(x_1)\big),
\end{array}
\end{equation}
where $P_i\in \R[s]$ is a monic polynomial of degree $d_i,$ with $d_2d_3\ne0,$ $P_1^2(\mathbf{0})+P_2^2(\mathbf{0})+P_3^2(\mathbf{0})\ne0,$ and $A_i(x_1)=a_{i}+b_{i}x_1$, $a_i,b_i\in\R,$ for $i=1,2,3$. This system is a \emph{nilpotent differential system} since its right-hand side is a nilpotent map \cite{CaEs2021}, that is, the linearization matrix of the differential system at each point of 
$\R^4$ is nilpotent. 

System \eqref{nilpotente} 
can be written in the form \eqref{main-system}
with the analytic functions
$$
\bm{\Phi}(\mathbf{x})=
\Big(x_1,x_2+A_1(x_1),
x_3+\frac{1}{d_2}A_2(x_1),x_4+\frac{1}{d_3}A_3(x_1)
\Big)
$$
and
$$ 
\mathbf{G}(\mathbf{u})=
\Big(P_1(u_2),P_2(u_3),P_3(u_4),0\Big)^T.
$$
Since 
$\mathbf{\Phi}(\mathbf{0})=\mathbf{0},$   $\mathbf{G}(\mathbf{0})=(P_1(\mathbf{0}),P_2(\mathbf{0}),P_3(\mathbf{0}),0)\ne\mathbf{0},$
and the reduced system 
$\dot{\mathbf{u}}=\mathbf{G}(\mathbf{u})$ is polynomially completely integrable, the differential system~\eqref{nilpotente} is polynomially completely integrable in $\R^4$ by Theorem~\ref{th:main}.
\end{example}

\begin{example}
\rm  {\bf A $n$-dimensional  family of Kolmogorov complete integrable systems.}	
Consider the polynomial functions $R=1,$
$$
\bm{\Phi}(\mathbf{x})=
\Big(x_1,x_1x_2,x_2x_3,\ldots,x_{n-1}x_n
\Big)
$$
and the $C^r\!$ function
$$ 
\mathbf{G}(\mathbf{u})=
\big(G_1(\mathbf{u}),1+G_2(\mathbf{u}),\ldots,1+G_{n-1}(\mathbf{u}),0\big),
$$
with $\mathbf{u}\in (\R^n,\mathbf{0})$ and $G_i(\mathbf{0})=0$ for $i\geq 2$.

From the identity 
$
 \mathbf {A} \operatorname {adj} (\mathbf {A} )=\det(\mathbf {A} )\mathbf{I}_n
$ and by defining 
$$
P_{ij}:=(-1)^{i+j}\prod_{\mu=1}^{j-2}x_{\mu}\prod_{\nu=j+1}^{n-1}x_{\nu},
$$
with $\prod_{\mu=1}^{j-2}x_{\mu}\equiv 1$ if $j \leq 2$ and $\prod_{\nu=j+1}^{n-1}x_{\nu}\equiv 1$ if $j\geq n-1$,
it follows that
$$
\operatorname {adj} (D\bm{\Phi}(\mathbf{x}))=
\begin{pmatrix}
x_1P_{11} & 0   & \cdots &  0 \\
x_2P_{21} &  x_2P_{22}  & \cdots & 0\\
\vdots & \vdots  & \ddots & \vdots   \\
x_nP_{n1}&  x_nP_{n2}& \cdots & P_{nn}
\end{pmatrix}.
$$
Hence,  system  \eqref{main-system} becomes
\begin{equation*}
\label{Kol-system-n}
\begin{array}{rcl}
\dot{x}_1 \!\!\!&=&\!\!\!
x_1P_{11}G_1(\bm{\Phi}(\mathbf{x})),
\\[0.4pc]
\dot{x}_2 \!\!\!&=&\!\!\!
x_2\Big(P_{21}G_1(\bm{\Phi}(\mathbf{x}))
+P_{22}\big(1+G_1(\bm{\Phi}(\mathbf{x}))\big)\Big),
\\[0.4pc]
&\vdots&\\[0.4pc]
\dot{x}_n \!\!\!&=&\!\!\! x_n\Big(P_{n1}G_1(\bm{\Phi}(\mathbf{x}))
+\sum_{j=2}^{n-1}P_{nj}\big(1+G_j(\bm{\Phi}(\mathbf{x}))\big)\Big),
\end{array}
\end{equation*}
which is a $C^r\!$ Kolmogorov system in $(\R^n,\mathbf{0})$.
Since 
$\Phi(\mathbf{0})=\mathbf{0}$
and
$\mathbf{G}(\mathbf{0})\neq \mathbf{0}$, this  system
is $C^r\!$ completely integrable  in $(\R^n,\mathbf{0})$ by Theorem~\ref{th:main}.
\end{example}

\section*{Acknowledgements}

The first author is supported by 
Direcci\'on de Investigaci\'on of the UCSC through project DIREG 01/2024. 
The second author is supported by the Spanish State Research Agency, through the project PID2022-136613NB-I00 grant and the   grant 2021-SGR-00113 from AGAUR, Generalitat de Ca\-ta\-lu\-nya.
The third author is supported by European research project H2020-MSCA-RISE-2017–777911 and Universidad del B\'{\i}o-B\'{\i}o grant RE2320122.

\end{document}